\documentclass[a4paper,12pt]{article}

\usepackage{hyperref}
\usepackage{graphicx,color,xcolor}
\usepackage{amsthm,amsmath,amssymb,amsfonts}
\usepackage[numbers, sort & compress]{natbib}

\def\pref#1{(\ref{#1})}

\makeatletter
\def\p@enumii{}
\makeatother

\newtheorem{thm}{Theorem}[section]
\newtheorem{lem}[thm]{Lemma}
\newtheorem{prop}[thm]{Proposition}
\newtheorem{cor}[thm]{Corollary}
\theoremstyle{definition}
\newtheorem{defn}[thm]{Definition}

\theoremstyle{remark}

\numberwithin{equation}{section}

\newcommand{\V}{\mathrm{V}}
\newcommand{\E}{\mathrm{E}}
\newcommand{\N}{\mathrm{N}}
\renewcommand{\L}{\mathrm{L}}
\renewcommand{\b}[1]{\overline{#1}}
\newcommand{\rg}{\rangle}
\renewcommand{\lg}{\langle}
\newcommand{\se}{\subseteq}
\newcommand{\link}{\mathrm{link}}
\newcommand{\sm}{\setminus }
\renewcommand{\iff}{\Leftrightarrow}
\newcommand{\give}{$\Rightarrow$}
\newcommand{\rgive}{$\Leftarrow$}
\newcommand{\ifof}{if and only if }
\newcommand{\tohi}{\emptyset}
\newcommand{\tl}[1]{\widetilde{#1}}
\newcommand{\rnd}{\partial}
\newcommand{\blt}{\bullet}
\newcommand{\im}{\mathrm{Im}\,}

\title{Algebraic Properties of Clique Complexes of Line Graphs}
\author{Ashkan Nikseresht \\
\it\small Department of Mathematics, College of Science, Shiraz University,\\
\it\small 71457-13565, Shiraz, Iran\\
\it\small E-mail: ashkan\_nikseresht@yahoo.com}
\date{}

\begin{document}
\maketitle

\begin{abstract}
Let $H$ be a simple undirected graph and $G=\L(H)$ be its line graph. Assume  that $\Delta(G)$ denotes the clique
complex of $G$. We show that $\Delta(G)$ is sequentially Cohen-Macaulay if and only if it is shellable if and only if
it is vertex decomposable. Moreover if $\Delta(G)$ is pure, we prove that these conditions are also equivalent to
being strongly connected. Furthermore, we state a complete characterizations of those $H$ for which $\Delta(G)$ is
Cohen-Macaulay, sequentially Cohen-Macaulay or Gorenstein. We use these characterizations to present linear time
algorithms which take a graph $G$, check whether $G$ is a line graph and if yes, decide if $\Delta(G)$ is
Cohen-Macaulay or sequentially Cohen-Macaulay or Gorenstein.
\end{abstract}


\section{Introduction}
In this paper, $K$ denotes a field and $S=K[x_1,\ldots, x_n]$. It is known that using several transformations on a
graded ideal $I$ of $S$, such as taking generic initial ideal, polarization, \ldots, we can get a square-free monomial
ideal $J$ generated in degree 2, such that $S/I$ is Cohen-Macaulay (CM for short) \ifof $S/J$ is so (see \cite{hibi}
and also \cite{bary}). But then $J$ is the edge ideal of a graph and this shows why it is important to study algebraic
properties of edge ideals of graphs. Let $G$ be a simple graph on vertex set $\V(G)=\{v_1,\ldots, v_n\}$ and edge set
$\E(G)$. Then the \emph{edge ideal} $I(G)$ of $G$ is the ideal of $S$ generated by $\{x_ix_j|v_iv_j\in \E(G)\}$. A
graph $G$ is called CM (resp. Gorenstein) when $S/I(G)$ is CM (resp. Gorenstein) for every field $K$. Many researchers
have tried to combinatorially characterize CM or Gorenstein graphs in specific classes of graphs, see for example,
\cite{obs to shell,tri-free, alpha=3, large girth,planar goren,CM circulant,Goren circulant, very well, hibi, CWalker,
Trung const}).

The family of cliques of a graph $G$ forms a simplicial complex which is called the \emph{clique complex of $G$} and is
denoted by $\Delta(G)$. Algebraic properties of simplicial complexes in general also has got a wide attention recently,
see for example \cite{bary,hibi,my vdec,our chordal,stanley} and the references therein. If we denote the
Stanley-Reisner ideal of $\Delta$ by $I_\Delta$, then we have $I_{\Delta(G)}=I(\b G)$, where $\b G$ denotes the
complement of the graph $G$. Thus studying clique complexes of graphs algebraically, is another way to study algebraic
properties of graphs.

Suppose that $H$ is a simple undirected graph and $G=\L(H)$ is the \emph{line graph} of $H$, that is, edges of $H$ are
vertices of $G$ and two vertices of $G$ are adjacent if they share a common endpoint in $H$. Line graphs are well-known
in graph theory and have many applications (see for example \cite[Section 7.1]{west}). In particular, Theorems 7.1.16
to 7.1.18 of \cite{west}, state some characterizations of line graphs and methods that, given a line graph $G$, can
find a graph $H$ for which $G=\L(H)$. Indeed, in \cite{lehot} a linear time algorithm is presented that, given a graph
$G$, it checks if $G$ is a line graph and if $G$ is a line graph, it returns a graph $H$ such that $G=\L(H)$.

Here we study some algebraic properties of $\Delta(G)$. Because of the aforementioned results and algorithm, we do this
in terms of the graph $H$ for which $G=\L(H)$. First in Section 2, we characterize combinatorially those $H$ with a CM
or a Gorenstein $\Delta(\L(H))$. As we will see, the class of such graphs is very limited.

Then in Section 3, we characterize those $H$ whose line graph has a sequentially CM clique complex. Recall that a
simplicial complex $\Delta$ is called sequentially CM, when its pure skeletons $\Delta^{[i]}=\lg F\in \Delta \big|
|F|=i+1 \rg$ are CM for all $i$ (see also \cite{stanley}, for an equivalent algebraic definition). Our characterization
enables us to present a linear time algorithm which decides whether $\Delta(\L(H))$, for a given graph $H$, is
sequentially CM or not.  Thus our algorithm enables us to efficiently decide whether a given graph is a line graph and
if yes, whether its clique complex is sequentially CM.

Here we say a simplicial complex $\Delta$ is CM (resp. Gorenstein) over $K$, when $S/I_\Delta$ is CM (resp.
Gorenstein). If $\Delta$ is CM (resp. Gorenstein) over every field $K$, then we simply say that $\Delta$ is CM (resp.
Gorenstein). For definitions and basic properties of simplicial complexes and graphs one can see \cite{hibi} and
\cite{west}, respectively. In particular, all notations used in the sequel without stating the definitions are as in
these two references.

                     \section{Line graphs with CM or Gorenstein clique complexes}

In what follows, $H$ is a simple undirected graph with at least one edge and $G=\L(H)$ is the line graph of $H$. We are
going to study when $\Delta(G)$ is CM. It is well-known that a CM complex is pure (see \cite[Lemma 8.1.5]{hibi}). The
following lemma characterizes those $H$ with pure $\Delta(G)$. Here $K_n$ denotes the complete graph on $n$ vertices.
Also we call a set of $r$ edges of $H$ adjacent to a common vertex $v$, an \emph{$r$-star} (or simply, a \emph{star})
of $H$ at $v$.

\begin{lem}\label{Del pure}
If $H$ is connected, then the clique complex of $G$ is pure, \ifof one of the following holds.
\begin{enumerate}
\item \label{Del pure 1} $H$ has no triangles and there is an integer $r>3$ such that every vertex of $H$ has degree
    either one or $r$.

\item \label{Del pure 2} The maximum degree of vertices of $H$ is 3 and every vertex of $H$ with degree 2 is
    contained in a triangle.

\item \label{Del pure 3} $H$ is a path or a cycle.
\end{enumerate}
\end{lem}
\begin{proof}
Note that since $H$ is connected, $G$ is also connected. Moreover, each clique of $G$ is either a set of edges sharing
an endpoint $v$ in $H$ , that is, a star of $H$ at $v$, or a triangle of $H$. Thus if $\Delta(G)$ is pure and $\dim
\Delta(G)>2$, then each star of $H$ should be contained in a star with size $\dim \Delta(G) +1$ and $H$ cannot have any
triangles. Thus case \pref{Del pure 1} occurs. If $\dim \Delta(G)=2$, then each star of $H$ should be in a star of size
$3$ or a triangle, that is, $H$ is as in case \pref{Del pure 2}. Finally, if $\dim \Delta(G)<2$, then vertices of $H$
have degree at most 2 and case \pref{Del pure 3} happens. The converse is clear.
\end{proof}

If $\Delta$ is pure and for any two facets $F$ and $G$ of $\Delta$, there is a sequence $F=F_1, \ldots, F_t=G$ of
facets of $\Delta$, such that $|F_i\cap F_{i+1}|= |F_i|-1$ for all $i$, we say that $\Delta$ is \emph{strongly
connected} (or connected in codimension 1). By \cite[Lemma 9.1.12]{hibi}, every CM complex is strongly connected. Thus
we next investigate when $\Delta(G)$ is strongly connected. Note that every strongly connected complex with $\dim>0$ is
connected. In the following, we call two faces $F_1$ and $F_2$ of $\Delta(G)$, \emph{adjacent} when $|F_1\cap
F_2|=|F_1|-1=|F_2|-1$, and by a \emph{strong path} between facets $F_1$ and $F_t$, we mean a sequence $F_1, \ldots,
F_t$ of facets, with $F_i$ adjacent to $F_{i+1}$ for all $1\leq i<t$.
\begin{lem}\label{Del st conn}
Suppose that $\Delta(G)$ is pure. Then $\Delta(G)$ is strongly connected \ifof $H$ is either a star or a path or a
cycle or one of the graphs in Figure \ref{fig1}.
\begin{figure}
\begin{center}
\includegraphics[scale=0.8]{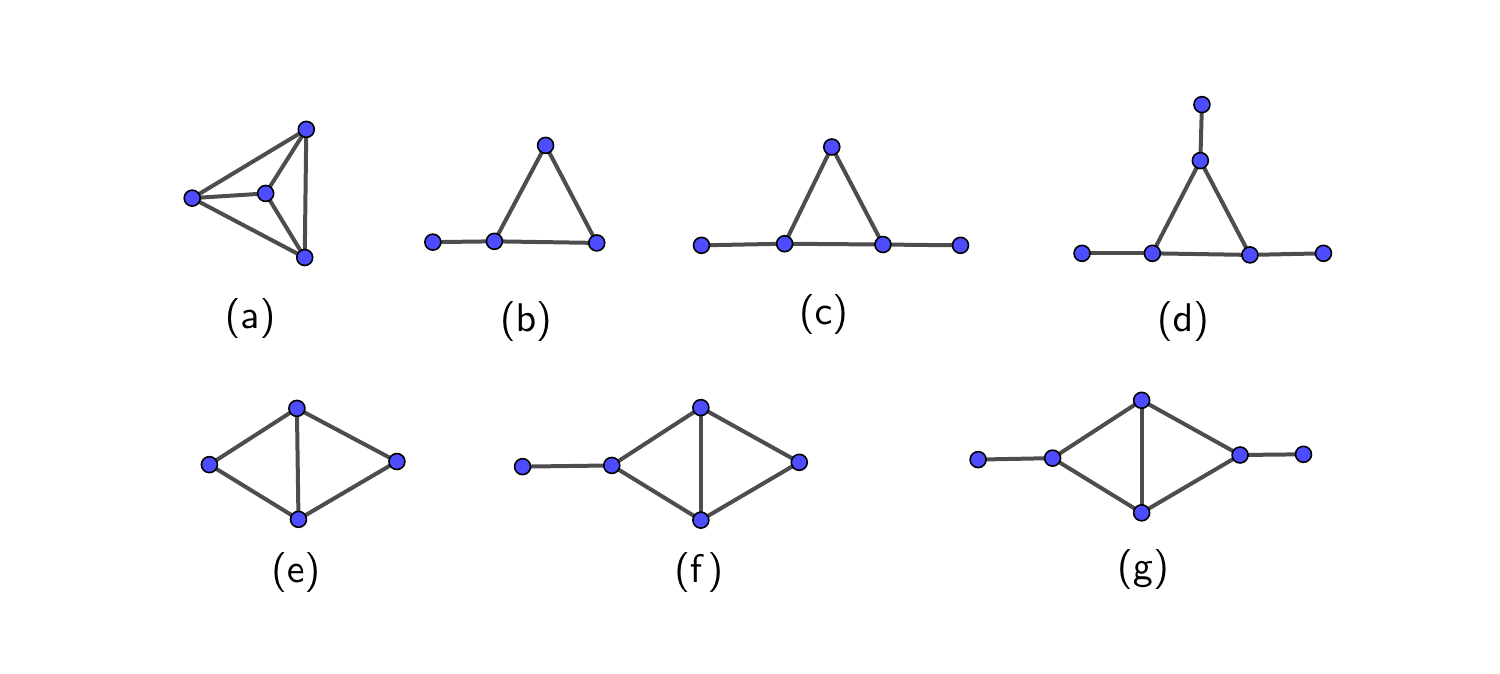}
\caption{Graphs with line graphs having strongly connected clique complexs \label{fig1}}
\end{center}
\end{figure}
\end{lem}
\begin{proof}
It is easy to check that the line graph of each of the mentioned graphs has a strongly connected clique complex.
Conversely, suppose that $\Delta(G)$ is strongly connected and $d=\dim \Delta(G)$. Note that two stars of $H$ at
different vertices can have at most one edge in common. Also no two different triangles are adjacent and a triangle can
only be adjacent with a 3-star at a vertex of the triangle. So if $d>2$ and $H$ has a vertex $v$ with degree $d+1$, the
$(d+1)$-star at $v$ cannot be adjacent to any other facet of $\Delta(G)$ and hence should be the only facet of
$\Delta(G)$. Therefore, by \pref{Del pure}, all other vertices of $H$ have degree 1 and $H$ is a star. If $d=1$, then
connectedness and strongly connectedness for $\Delta(G)$ are equivalent and by \pref{Del pure}, $H$ is either a path or
a cycle.

Now assume that $d=2$. If $H$ has no triangle, then similar to the case that $d>2$, $H$ is a star. If $H$ has just one
triangle, then every 3-star of $H$ must be adjacent to this triangle, which means, should be centered at a vertex of
the triangle. Thus $H$ is either a triangle, that is, a 3-cycle or one of the graphs (b), (c) or (d) of Figure
\ref{fig1}. Now assume that $H$ has exactly 2 triangles $T_1$ and $T_2$. Then there should be a 3-star $F$, such that
both $T_1$ and $T_2$ are adjacent to $F$. Thus $F$ is centered at a vertex of both $T_1$ and $T_2$ and must share two
edges with each of them. Since degree of each vertex is at most three, it follows that $T_1$ and $T_2$ have a common
edge. Again as every 3-star of $H$ is adjacent to either $T_1$ or $T_2$, $H$ should be one of the graphs (e), (f) or
(g) in Figure \ref{fig1}.

Note that if $H$ has more than 2 triangles, again a 3-star should be adjacent to two triangles and hence $H$ has a
subgraph isomorphic to the graph of Figure \ref{fig1}(e). Lets call this subgraph $H_0$, call the vertices of degree 2
of $H_0$, $a$ and $b$ and call the vertices of degree 3 of $H_0$, $u$ and $v$. Suppose that $H$ has a triangle not
contained in $H_0$. Then there exists a strong path starting with a triangle in $H_0$ and ending with a triangle not in
$H_0$. Let $T$ be the first triangle in this path which is not in $H_0$. Then $T$ and a triangle of $H_0$ are both
adjacent to a 3-star and hence share an edge. But any edge in $H_0$ in incident to either $u$ or $v$, hence $u$ or $v$
is in $T$. Thus $T$ is either $uab$ or $vab$. This means that $a$ and $b$ are adjacent and since all vertices have
degree at most 3, $H$ is the graph in Figure \ref{fig1}(a).
\end{proof}

Recall that for a face $F$ of a simplicial complex $\Delta$, we define $\link_\Delta F= \{G\sm F|F\se G\in \Delta\}$.
Also for a vertex $v$ of $\Delta$, $\Delta-v$ is the simplicial complex with faces $\{F\in \Delta|v\notin F\}$. A
vertex $v$ of a nonempty simplicial complex $\Delta$ is called a \emph{shedding vertex}, when no face of
$\link_\Delta(v)$ is a facet of $\Delta-v$. By \cite[Lemma 3.1]{my vdec}, if $\Delta$ is pure, a vertex $v$ is a
shedding vertex \ifof $\Delta-v$ is pure and $\dim(\Delta-v)=\dim \Delta$. A nonempty simplicial complex $\Delta$ is
called \emph{vertex decomposable}, when either it is a simplex or there is a shedding vertex $v$ such that both
$\link_\Delta v$ and $\Delta- v$ are vertex decomposable. The $(-1)$-dimensional simplicial complex $\{\emptyset\}$ is
considered a simplex and hence vertex decomposable. To characterize line graphs with pure vertex decomposable clique
complexes, we use the following lemmas.

\begin{lem}\label{graph ver dec}
Suppose that $C$ is a graph without isolated vertices. Then $\Delta=\lg \E(C) \rg$ is vertex decomposable if and only
if $C$ is connected.
\end{lem}
\begin{proof}
Note that every 0-dimensional simplicial complex is vertex decomposable. Thus $\Delta$ is vertex decomposable \ifof it
can be transformed to a simplex by repeatedly deleting shedding vertices. Also a vertex $v$ of $\Delta$ is a shedding
vertex, \ifof $v$ has no neighbor with degree 1. Assume that $C$ is connected. If $C$ has a vertex $v$ of degree 1,
then $v$ is a shedding vertex unless $C$ is $K_2$, in which case $\Delta$ is a simplex and vertex decomposable. Thus we
can delete $v$ and a get a smaller connected graph, hence the result follows by induction. If $C$ has no vertex of
degree 1, then there is a (shedding) vertex $v$ of $C$ such that $C-v$ is again connected. Again the result follows by
induction.

Conversely, deleting a shedding vertex from $C$ does not decrease the number of connected components of $C$. Hence if
$C$ is not connected, then after deleting any number of shedding vertices, the obtained graph $C'$ is still not
connected and hence $\lg \E(C') \rg$ is not a simplex. Hence $\Delta$ is not vertex decomposable.
\end{proof}

It should be mentioned that \cite[Lemma 3.1]{our chordal}, states that for a graph $C$, the simplicial complex $\lg
\E(C) \rg$ is ``vertex decomposable'' \ifof $C$ is a tree, which is in contradiction with the above lemma. This is
because, as mentioned in a corrigendum to \cite{our chordal}, vertex decomposability  as used in \cite{our chordal},
differs slightly with that used in the literature and here. For more details, see the corrigendum at the end of the
arXiv version of \cite{our chordal}.

\begin{lem}\label{inde graph ver dec}
If each connected component of a graph $C$ is a tree, then $\Delta(\b C)$ is vertex decomposable.
\end{lem}
\begin{proof}
Immediate consequence of \cite[Theorem 1]{obs to shell}.
\end{proof}

Another combinatorial property which is stronger than being CM is shellability. If there is an ordering $F_1, \ldots,
F_t$ of all facets of $\Delta$ such that for each $i$ we have $\lg F_1, \ldots, F_i \rg \cap \lg F_{i+1} \rg$ is a pure
simplicial complex of dimension $= \dim F_{i+1}-1$, then $\Delta$ is called \emph{shellable} and such an order is
called a \emph{shelling order}. It is well-known that a vertex decomposable complex is shellable (see for example
\cite[Section 2]{obs to shell}) and a pure shellable complex is CM (\cite[Theorem 8.2.6]{hibi}).

\begin{thm}\label{Del main}
Suppose that $H$ is a graph with at least one edge and $G=\L(H)$. Then the following are equivalent for
$\Delta=\Delta(G)$.
\begin{enumerate}
\item \label{Del main 1} $\Delta$ is pure vertex decomposable.
\item \label{Del main 2} $\Delta$ is pure shellable.
\item \label{Del main 3} $\Delta$ is CM (over some field).
\item \label{Del main 4} $\Delta$ is pure and strongly connected.
\item \label{Del main 5} $H$ is either a star or a path or a cycle or one of the graphs in Figure \ref{fig1}.
\end{enumerate}
\end{thm}
\begin{proof}
The facts that \pref{Del main 1} \give\ \pref{Del main 2} \give\ \pref{Del main 3} \give\ \pref{Del main 4} are
well-known. Also \pref{Del st conn} shows that \pref{Del main 4} $\iff$ \pref{Del main 5}.

\pref{Del main 5} \give\ \pref{Del main 1}: It is easy to check that in each case $\Delta$ is pure. If $H$ is a star,
then $G$ is a complete graph and $\Delta$ is a simplex. If $H$ is a path or a cycle, then $G$ is also a path or a cycle
and $\Delta(G)= \lg \E(G) \rg$ (unless $H=K_3$, in which case  $\Delta$ is a simplex). Consequently, the result follows
from \pref{graph ver dec}. It is routine to check that if $H$ is any of the graphs in Figure \ref{fig1}, except the
graphs (d) and (g), then each connected component of $\b G$ is a tree and according to \pref{inde graph ver dec},
$\Delta=\Delta(\b{\b G})$ is vertex decomposable.

Now assume that $H$ is the graph (d) in Figure \ref{fig1}. Then $G$ is the graph (a) in Figure \ref{fig2}. If $v$ is
the vertex specified in Figure \ref{fig2}(a), then  $\Delta(G)-v=\Delta(G-v)$ is pure with dimension 2. Hence $v$ is a
shedding vertex of $\Delta$ and since $\b{G-v}$ is a tree, it follows that $\Delta-v$ is vertex decomposable. Also
$\link_\Delta(v)$ is a simplex. Thus, $\Delta$ is vertex decomposable. If $H$ is the graph (g) in Figure \ref{fig1},
then $G$ is the graph (b) in Figure \ref{fig2} and a similar argument, with $v$ as specified in Figure \ref{fig2}(b),
shows that $\Delta$ is vertex decomposable.
\end{proof}
\begin{figure}
\begin{center}
\includegraphics[scale=1.5]{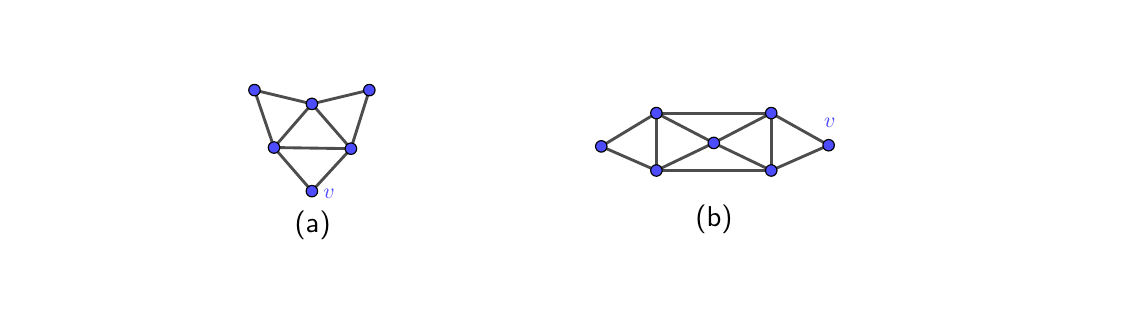}
\caption{The graph $G$ when $H$ is the graph (a) in Figure (\ref{fig1}d) (b) in Figure (\ref{fig1}g) \label{fig2}}
\end{center}
\end{figure}

Now we can also answer the question ``when $\Delta(G)$ is Gorenstein?'' Note that this question is equivalent to asking
when $\b G$ is a Gorenstein graph. By \cite[Lemma 3.1]{large girth} or \cite[Lemma 3.5]{tri-free}, each Gorenstein
graph $C$ without isolated vertices is a \emph{W$_2$ graph}, that is, $|\V(C)|\geq 2$ and every pair of disjoint
independent sets of $C$ are contained in two disjoint maximum independent sets of $C$.

\begin{cor}\label{Del Gor}
Assume that $H$ is a graph with at least one edge and $G=\L(H)$. Then $\Delta(G)$ is Gorenstein \ifof $H$ is either a
star or a cycle or a path with length at most 3 or one of the graphs in Figure \ref{fig3}.
\end{cor}
\begin{proof}
Since every Gorenstein complex is CM, we must search for graphs with Gorenstein $\Delta(G)$ between the graphs
mentioned in \pref{Del main}\pref{Del main 5}. Since a simplex is Gorenstein and also the complement of every cycle is
Gorenstein by \cite[Corollary 2.4]{alpha=3}, we deduce that if $H$ is a star or a cycle, then $\Delta(G)$ or
equivalently $\b G$ is Gorenstein. If $H$ is a path of length $n$, then $G$ is a path of length $n-1$ and if $n-1\geq
3$, so by \cite[Corollary 2.4]{alpha=3}, $\b G$ is not Gorenstein. If $n-1<3$, then each connected component of $\b G$
is $K_2$ or $K_1$ and hence $\b G$ is Gorenstein.

If $H$ is one of the graphs in Figure \ref{fig1} and not in Figure \ref{fig3}, except for the case that $H$ is the
graph (g) of Figure \ref{fig1}, then $\b G$ has a vertex of degree 1 in a non-complete connected component. But
according to Theorem 4 of \cite{W2} such a graph is not W$_2$ and hence $\b G$ is not Gorenstein. If $H$ is the graph
(g) in Figure \ref{fig1}, then the smallest cycle of $\b G$ has length 5. But by \cite[Theorem 7]{pinter2}, every W$_2$
graph with girth at least five is either the 5-cycle or $K_2$. Thus in this case also $\b G$ is not W$_2$ nor
Gorenstein. If $H$ is any of the graphs in \ref{fig3}, then every connected component of $\b G$ is $K_2$ or $K_1$ and
hence $\b G$ is Gorenstein.
\end{proof}
\begin{figure}
\begin{center}
\includegraphics[scale=2]{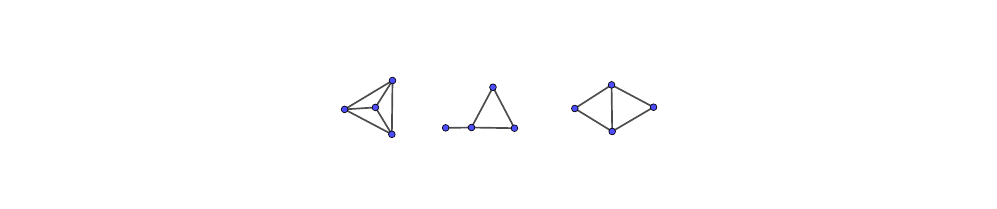}
\caption{Some graphs with Gorenstein $\Delta(G)$. \label{fig3}}
\end{center}
\end{figure}

At the end of this section, we should mention that we can use \pref{Del main} and \pref{Del Gor} to decide whether a
given graph $G$ is a line graph with a CM or a Gorenstein clique complex or not. For Cohen-Macaulayness, we must check
whether $G$ is a complete graph or a path or a cycle or the line graph of one of the  graphs depicted in Figure
\ref{fig1}. And for being Gorenstein, $G$ must be either a complete graph, a cycle, a path length $\leq 2$ or the line
graph of one of the graphs in Figure \ref{fig3}. It is clear that this can be carried out with linear time complexity.

                     \section{Line graphs with sequentially CM clique complexes}

In this section, we consider the case that $\Delta(G)$ is not pure and characterize those $H$ whose line graphs have
sequentially CM clique complexes. Recall that $\Delta^{[i]}=\lg F|F\in \Delta, \dim F=i\rg$ is called the \emph{pure
$i$-skeleton} of $\Delta$ and if each $\Delta^{[i]}$ is CM for $i\leq \dim \Delta$, then $\Delta$ is called
\emph{sequentially CM}. Note that every 0-dimensional complex is CM and a pure 1-dimensional complex is CM \ifof it is
connected (see for example \cite[Exercise 5.1.26]{CM ring}). The following result considers $\Delta^{[i]}$ for $i\geq
3$. In this section, we always assume that $\Delta=\Delta(G)$.

\begin{prop}\label{1 deg>3}
Suppose that $H$ is connected. Then all nonempty $\Delta^{[i]}$ for $i\geq 3$ are CM \ifof $H$ has at most one vertex
$v$ with degree $\geq 4$.
\end{prop}
\begin{proof}
If $H$ has no vertex with degree $\geq 4$, then $\Delta^{[i]}=\tohi$ for $i\geq 3$. If $H$ has exactly one vertex with
degree $r\geq 4$, then for $3\leq i< r$, $\Delta^{[i]}$ is the pure $i$-skeleton of the simplex with the $r$-star at
$v$ as the only facet. Hence $\Delta^{[i]}$ is CM. If $v_1\neq v_2$ are two vertices of $H$ with degree $\geq 4$ and
$E$ and $F$ are 4-stars of $H$ at $v_1$ and $v_2$, respectively, then $E$ and $F$ are facets of $\Delta^{[4]}$. Note
that every facet of $\Delta^{[4]}$ which is adjacent to $E$, is a 4-star at $v_1$, hence $\Delta^{[4]}$ is not strongly
connected and hence not CM.
\end{proof}

Next we are going to introduce a graph $H'$, such that $\Delta(\L(H'))^{[2]}=\Delta^{[2]}$ and in $H'$ every vertex of
degree 2 is in a triangle. For this we need the following lemmas. It should be mentioned that in this paper, just one
side of the following lemmas are used, but we state and prove both sides for completeness. The first lemma is easy and
its proof is left to the reader.
\begin{lem}\label{isolated}
If $\Gamma$ is a simplicial complex with an isolated vertex $v$, then $\Gamma$ is vertex decomposable  (resp.
shellable, sequentially CM) \ifof $\Gamma - v$ is so.
\end{lem}

\begin{lem}\label{add new ver}
Suppose that $\Gamma$ is a connected simplicial complex and $a$ is a vertex of $\Gamma$. Also assume that $b$ is a new
vertex. Then $\Gamma$ is vertex decomposable (resp. shellable, sequentially CM)  \ifof $\Gamma+\lg ab \rg$ is so.
\end{lem}
\begin{proof}
(\give): If $\Gamma=\lg a \rg$, then $\Gamma'= \Gamma+ \lg ab \rg$ is a simplex and the result is clear. Assume that
$\Gamma \neq \lg a \rg$, that is, $\{a\}$ is not a facet of $\Gamma$. Then $b$ is a shedding vertex of $\Gamma'$ with
$\link_{\Gamma'} (b)=\lg a \rg$ and $\Gamma'- b= \Gamma$. Hence if $\Gamma$ is vertex decomposable, then $\Gamma'$ is
also vertex decomposable. If $\Gamma$ is shellable, then it can readily be checked that by adding $\{ab\}$ at the end
of a shelling order of $\Gamma$, we get a shelling order of $\Gamma'$. The case for sequentially CM follows from the
facts that $\Gamma^{[i]}=\Gamma'^{[i]}$ for all $i> 1$ and for $i=1$, being CM is equivalent to being connected.

(\rgive):  For sequentially CM, the proof again follows the aforementioned facts. Suppose that $\Gamma'$ is shellable.
In any shelling order of $\Gamma'$, either $\{ab\}$ is the first facet or there is a facet $E$ containing $a$ before
$\{ab\}$. In the former case, the second facet should be a facet $E$ containing $a$. Thus in both cases for each term
$F_i$ after $\{ab\}$ in the shelling order, $F_i \cap \{ab\} \se F_i\cap E$ and it follows that dropping $\{ab\}$ from
a shelling order of $\Gamma'$, gives us a shelling order for $\Gamma$.

Now assume that $\Gamma'$ is vertex decomposable. If $\Gamma'$ is a simplex then we must have $\Gamma=\lg a\rg$ and is
vertex decomposable. So suppose that $v$ is a shedding vertex of $\Gamma'$ such that both $\link_{\Gamma'}(v)$ and
$\Gamma'- v$ are vertex decomposable. If $v=b$, then $\Gamma=\Gamma'- v$ is vertex decomposable. Also $v\neq a$,
because $\{b\}$ is a facet of both $\link_{\Gamma'}(a)$ and $\Gamma'- a$. Thus we assume $v\neq a, b$. Now
$\link_{\Gamma}(v)=\link_{\Gamma'}(v)$ and $\Gamma- v+ \lg ab \rg=\Gamma'- v$. Thus by induction we can deduce that
$\Gamma- v$, as well as $\link_\Gamma(v)$, is vertex decomposable. Therefore, if $v$ is a shedding vertex of $\Gamma-
v$, we are done. Else, a facet $F$ of $\link_{\Gamma}(v)$ is a facet of $\Gamma- v$ but not a facet of $\Gamma - v+\lg
ab \rg$. Thus $F$ must be $\{a\}$ and in $\Gamma'- v$,  $a$ is only connected to $b$. Since $\Gamma'- v$ is vertex
decomposable and hence sequentially CM, its pure 1-skeleton should be CM and connected. This means that $\Gamma' - v$
is just a set of isolated vertices along with $\lg ab \rg$ or equivalently facets of $\Gamma$ are either 0-dimensional
or 1-dimensional containing $v$. Now the result follows from Lemmas \ref{isolated} and \ref{graph ver dec}.
\end{proof}

Recall that a \emph{free vertex} of a simplicial complex is a vertex which is contained in exactly one facet.
\begin{lem}\label{free vertices vdec}
Suppose that $\Gamma$ is a simplicial complex and $a,b$ are two free vertices of $\Gamma$ contained in the facets $E$
and $F$, respectively. Also assume that $F\neq E$ and $|E|, |F|\geq 2$. If $\Gamma$ is vertex decomposable (resp.,
shellable, sequentially CM), then $\Gamma+\lg ab \rg$ is so. Moreover if $|E|, |F|\geq 3$, then the converse also
holds.
\end{lem}
\begin{proof}
(\give): Set $\Gamma'=\Gamma+\lg ab \rg$. We use induction on the number of vertices of $\Gamma$. Suppose that $\Gamma$
is vertex decomposable and $v$ is a shedding vertex of $\Gamma$ with both $\link_\Gamma (v)$ and $\Gamma-v$ vertex
decomposable. If $v=a$, then $\link_{\Gamma'}(v)=\link_{\Gamma}(v) +\lg b \rg$ and $b$ is isolated in this link. Hence
by \pref{isolated} $\link_{\Gamma'}(v)$ is vertex decomposable. Also $\Gamma'-v=\Gamma-v$ and as $b$ is not a facet of
$\Gamma'-v$, $v$ is a shedding vertex of $\Gamma'$ and $\Gamma'$ is vertex decomposable. The case that $v=b$ is
similar. Suppose $v\neq a, b$. If $v\in E$, then $E\sm \{v\}$ is a facet of $\link_{\Gamma}(v)$ and hence there is a
facet $E'$ of $\Gamma-v$, such that $E\sm \{v\} \subsetneq E'$. So $E'\neq E$ is a facet of $\Gamma$ containing $a$,
contradicting freeness of $a$. Consequently, $v\notin E$ and similarly $v\notin F$ and hence the size of the facet
containing $a$ or $b$ does not differ in $\Gamma$ and $\Gamma-v$. Therefore, $\Gamma'- v= (\Gamma- v) +\lg ab \rg$ is
vertex decomposable by induction hypothesis and $\link_{\Gamma'}(v)=\link_{\Gamma}(v)$ is also vertex decomposable.
Clearly $v$ is a shedding vertex of $\Gamma'$ and hence $\Gamma'$ is vertex decomposable. The result for shellability
and being sequentially CM follows an argument similar to the proof of the previous lemma.

(\rgive): Now assume that $|E|, |F|\geq 3$. If $\Gamma'$ is shellable and in any of its shelling orders one of $E$ or
$F$, say $E$, is after $\{ab\}$, then the intersection of $\lg E \rg$ with the previous terms of the order has a facet
$\{a\}$ with $0=\dim \{a\}<\dim E -1$, a contradiction. So $\{ab\}$ appears  after both $E$ and $F$. Now if $F_i$ is a
term after $\{ab\}$, then $\lg F_i \rg \cap \lg ab \rg= \lg F_i \rg \cap \lg E, F\rg$. Hence by deleting the term
$\{ab\}$ from a shelling order of $\Gamma'$, we get a shelling order for $\Gamma$. If $\Gamma'$ is sequentially CM,
then as $\Gamma^{[i]}=\Gamma'^{[i]}$ for $i>1$, we just need to show that $\Gamma^{[1]}$ is CM or equivalently
connected. Because $\Gamma'^{[1]}=\Gamma^{[1]}+ \lg ab \rg$ is connected, it suffices to show that $a$ and $b$ are
connected in $\Gamma$. By our assumption, there are $a\in E_0\se E$ and $b\in F_0 \se F$, with $|E_0|=|F_0|=3$. Thus
there is a strong path between $E_0$ and $F_0$ in $\Gamma^{[2]}=\Gamma'^{[2]}$ and it follows that $a$ and $b$ are
connected in $\Gamma$.

Finally assume that $\Gamma'$ is vertex decomposable and $v$ is a shedding vertex of $\Gamma'$ with both
$\link_{\Gamma'}(v)$ and $\Gamma'-v$ vertex decomposable. If $v=a$, then $\link_{\Gamma'}(v)=\link_{\Gamma}(v) +\lg b
\rg$ and $b$ is isolated in this link. Hence according to \pref{isolated}, $\link_{\Gamma}(v)$ is vertex decomposable.
Also $\Gamma'-v=\Gamma-v$ and as $b$ is not a facet of $\Gamma-v$, $v$ is a shedding vertex of $\Gamma$ and $\Gamma$ is
vertex decomposable. The case that $v=b$ is similar. If $v\neq a, b$, then an argument similar to the proof of (\give),
shows that $v\notin E\cup F$. Thus $\link_\Gamma(v)=\link_{\Gamma'}(v)$ and $\Gamma'- v= \Gamma- v +\lg ab \rg$ is
vertex decomposable by induction. If $D$ is a facet of both $\link_\Gamma(v)$ and $\Gamma-v$, then it is strictly
contained in a facet of $\Gamma'-v$, since $v$ is shedding in $\Gamma'$. Thus we must have $D\subsetneq \{ab\}$ so
$D=\{a\}$ or $\{b\}$. But then $D$ is strictly contained in $E\sm \{v\}$ or $F\sm \{v\}$ which are facets of
$\link_\Gamma(v)$, a contradiction. Consequently, $v$ is a shedding face of $\Gamma$ and $\Gamma$ is vertex
decomposable.
\end{proof}

We should mention that the conditions $|F|, |G| \geq 3$ is necessary for the converse of this lemma, for example,
$\Gamma=\lg av, bu \rg$ is not sequentially CM but $\Gamma+\lg ab \rg$ is vertex decomposable.

Suppose that $v$ is a vertex of $H$ with degree 2 adjacent to vertices $a$ and $b$. By \emph{splitting} $v$, we get the
graph $H'$ with vertex set $(\V(H)\sm \{v\}) \cup \{v_1,v_2\}$, where $v_1$ and $v_2$ are new vertices, and the same
edge set as $H$, where we identify the edges $av$ and $bv$ of $H$ with $av_1$ and $bv_2$ in $H'$. Note that $v_1$ and
$v_2$ which we call \emph{the halves of $v$} are both leaves (vertices of degree 1) in $H'$. Here by a \emph{leaf edge}
we mean an edge incident to a leaf.
\begin{prop}\label{split}
Suppose that $H$ is connected. Then the following are equivalent.
\begin{enumerate}
\item \label{split 1} $\Delta(G)$ is sequentially CM.
\item \label{split 2} If $H'$ is obtained by splitting all vertices of degree 2 of $H$ which are not in a triangle,
    then every connected component of $H'$ is an edge except at most one component whose line graph has a
    sequentially CM clique complex.
\item \label{split 3} $H$ can be obtained by consecutively applying the following two operations on a graph $H_0$ in
    which every vertex of degree two is in a triangle and whose line graph has a sequentially CM clique complex:
    \begin{enumerate}
    \item \label{A} attaching a new leaf to an old leaf of the graph;
    \item \label{B} unifying two leaves whose distance is at least 4.
    \end{enumerate}
\end{enumerate}
Moreover, if any the above statements holds, $H_0$ is as in \pref{split 3} and $\Delta(\L(H_0))$ is vertex decomposable
(resp. shellable), then $\Delta(G)$ is vertex decomposable (resp. shellable).
\end{prop}
\begin{proof}
\pref{split 1} \give\ \pref{split 2}: We assume that $H$ and $H'$ have the same set of edges, according to the remarks
before the proposition. If $\dim \Delta = 1$,  then $H$ is a cycle or a path and hence either every component of $H'$
is an edge or $H'$ is a triangle. Assume that $\dim \Delta>1$. Suppose that $F_1$ and $F_2$ are adjacent facets of
$\Delta^{[2]}$. Then one of $F_1$ and $F_2$, say $F_1$ is a triangle and the other one, $F_2$, is a 3-star in $H$ with
two vertices (esp. the center of the star) on $F_1$. Note that the vertices of the triangle are not split in $H'$ and
hence $F_1$ and $F_2$ are on the same component of $H'$. As $\Delta^{[2]}$ is CM and hence strongly connected, it
follows that all triangles and all 3-stars of $H$ are in one component of $H'$, call it $H_0$. Noting that every vertex
with degree $r\geq 3$ is the center of $\binom{r}{3}$ 3-stars, we see that all vertices of degree $\geq 3$ of $H$ and
also all triangles of $H$ are on $H_0$. Hence $\Delta(\L(H_0))^{[i]}=\Delta^{[i]}$ for each $i\geq 2$ and so
$\Delta(\L(H_0))$ is sequentially CM. Since all vertices of $H'$ not on $H_0$ have degree at most 2 and $H'$ has no
vertex of degree 2 not on a triangle, all other components of $H'$ are edges.

\pref{split 2} \give\ \pref{split 3}: If every connected component of $H'$ is an edge, then $H$ is a path or a cycle
with length at least 4 and can be constructed by applying \pref{A} and \pref{B} starting with any of the connected
components of $H'$. Else let $H_0$ be the component of $H'$ which is not an edge. Then $\Delta(\L(H_0))$ is
sequentially CM and $H_0$ has no vertex of degree 2 not in a triangle (because all such vertices have been split
before). Thus we can construct $H$ from $H_0$ by ``undoing'' the vertex splits one by one. Note that when splitting
vertices of degree 2 not in a triangle, the order of the vertices to split does not make any difference on the final
graph. So we can assume that one half of the last vertex $v$ which is split is on $H_0$. If both $v_1$ and $v_2$, the
halves of $v$, are on $H_0$, then unifying $v_1$ and $v_2$ which are leaves, makes the last split undone. Note that as
$v$ was not in a triangle when it got split, the distance of $v_1$ and $v_2$ is at least 4. If only one half of $v$,
say $v_1$ is on $H_0$, then the other half is on an edge component of $H'$. Thus undoing the last split in this case is
indeed attaching a leaf to $v_1$. Hence ``undoing'' each vertex split is indeed applying \pref{A} or \pref{B}.

\pref{split 3} \give\ \pref{split 1} and the ``moreover'' statement: Let $G_0=\L(H_0)$ and $\Delta_0=\Delta(G_0)$. Note
that each leaf edge of $H_0$ is a free vertex of $\Delta_0$. Let $v_1 , v_2$ be leaves of $H_0$, $e_1,e_2$ be leaf
edges containing them and $E_1$ and $E_2$ be facets of $\Delta_0$ containing $e_1$ and $e_2$, respectively. If $v_1$
and $v_2$ have distance at least 4, then $e_1$ and $e_2$ neither are adjacent nor have a common neighbor in $G_0$ and
hence $E_1\neq E_2$. Also note that $|E_i|=1$ \ifof $G_0$ is just one edge. Since the operations mentioned in
\pref{split 3}, just add a new facet $\{ee'\}$ to $\Delta_0$ (for \pref{A}, $e$ is a free vertex of $\Delta_0$ and $e'$
a new vertex, and for \pref{B}, both of $e$ and $e'$ are free vertices of $\Delta_0$), the result follows from Lemmas
\ref{add new ver} and \ref{free vertices vdec}.
\end{proof}

An instant consequence of \pref{split} is the following characterization of graphs with no vertex of degree $\geq 4$
whose line graphs have sequentially CM clique complexes.
\begin{cor}\label{no deg>4}
Suppose that $H$ is a graph with at least one edge and $G=\L(H)$. If $\deg v\leq 3$ for each vertex $v$ of $H$, then
the following are equivalent.
\begin{enumerate}
\item $\Delta(G)$ is vertex decomposable.

\item $\Delta(G)$ is shellable.

\item $\Delta(G)$ is sequentially CM (over some field).

\item $H$ can be constructed from a star or a path or a cycle or a graph in Figure \ref{fig1}, by consecutively
    applying \pref{A} and \pref{B} of \pref{split}\pref{split 3}.
\end{enumerate}
\end{cor}
\begin{proof}
Let $H_0$ be a subgraph of $H$ in which every vertex of degree 2 is in a triangle. Then $\Delta(H_0)^{[i]}=\tohi$ for
$i>2$. Hence $\Delta(H_0)$ is pure and being sequentially CM is equivalent to being CM for $\Delta(H_0)$. Consequently,
the result follows by \pref{Del main} and \pref{split}.
\end{proof}


In the sequel, unless stated otherwise explicitly, we assume that $H_0$ is a connected graph with exactly one vertex
$v$ with degree $r>3$ and also suppose that every vertex of degree 2 in $H_0$ is in a triangle. We also let
$G_0=\L(H_0)$ and $\Delta_0=\Delta(G_0)$. According to \pref{split} and its corollary, by characterizing  those $H_0$
for which $\Delta_0$ is sequentially CM, we can derive a characterization of all graphs whose line graphs have a
sequentially CM clique complex. Noting that for $i>2$, $\Delta_0^{[i]}$ is either empty or the pure $i$-skeleton of a
simplex and for $i<2$, $\Delta_0^{[i]}$ is CM since $\Delta_0$ is connected, we just need to see when $\Delta_0^{[2]}$
is CM. First we study when $\Delta_0^{[2]}$ is strongly connected.

Suppose that $l_0=\{v\}$ and define $L_i=\N_{H_0}(L_{i-1})\sm (\cup_{j=0}^{i-1} L_j)$ to be the set of vertices of
\emph{level $i$} in $H_0$. Here $\N_{H_0}(A)$ is the set of all vertices adjacent to a vertex in $A$ inside the graph
$H_0$. Thus indeed, the level of a vertex is its distance to $v$. Note that a vertex with level $i$ can be adjacent
only to vertices with levels $i-1, i, i+1$. Suppose that $H_0[L_i]$ is the induced subgraph of $H_0$ on the vertex set
$L_i$. Then if $H'=H[L_1]$, every $u\in L_1$ has degree at most 2 in $H'$, since it is also adjacent to $v$ in $H_0$.
Therefore each connected component of $H'$ is either an isolated vertex or a cycle or a path of length $\geq 1$. We
call these isolated vertices, cycles and paths with positive lengths of $H[L_1]$, the \emph{level 1 isolated vertices},
\emph{level 1 cycles} and \emph{level 1 paths}, respectively.

\begin{lem}\label{lem st con}
If $\Delta_0^{[2]}$ is strongly connected, then every vertex $x$ with level $\geq 2$ and $\deg(x)=3$, has level 2 and
is adjacent to both endpoints of a level 1 path with length 1.
\end{lem}
\begin{proof}
Let $F$ be the 3-star at $x$ and $F'$ be a 3-star at $v$. Then as $\Delta_0^{[2]}$ is strongly connected, there exist a
strong path $F=F_1, \ldots, F_k=F'$ of 3-stars and triangles of $H_0$. We suppose that this strong path is the smallest
possible and in particular, there is no repetition in the path. Assume that $t$ is the smallest index such that $v$ is
incident to an edge in $F_t$. Thus $F_t$ is a 3-star at a level 1 vertex $a$ and $F_t=\{av, ab,ac\}$. Since $v$ is not
on any edge of $F_{t-1}$, $F_{t-1}$ must be the triangle $\{ab,ac,bc\}$ and $F_{t-2}$ is a 3-star at $b$ or $c$. As
each 3-star at a level 1 vertex has an edge incident to $v$, the center of $F_{t-2}$ has level 2. Also since $F_{t+1}$
is a triangle and $|F_{t+1}\cap F_t|=2$ is $F_{t+1}$ is the edge set of either the triangle $vab$ or the triangle
$vac$. In particular, at least one of $b,c$ has level 1. So we can suppose that $b\in L_1$ and $c\in L_2$ and
$F_{t-2}=\{cb,ca,cd\}$ is the 3-star at $c$. If $F_{t-2}\neq F$, then $t\geq 4$ and $F_{t-3}$ must be either
$\{cb,cd,db\}$ or $\{ca,cd, da\}$ for a vertex $d\notin \{a,b,c,v\}$. But then $a$ or $b$ has degree $>3$, a
contradiction. Thus $F_{t-2}=F$ and hence $x=c$ has level 2. Moreover, $x$ is adjacent to $a,b$ which have level 1 and
also are adjacent to each other. So $ab$ is a level 1 path with length 1, as required.
\end{proof}

\begin{prop}\label{st conn Del2}
The complex $\Delta_0^{[2]}$ is strongly connected, \ifof $H_0$ satisfies both of the following conditions (see an
example in Figure \ref{fig6}).
\begin{enumerate}
\item \label{st conn Del2 1} Every level 3 vertex of $H_0$ is a leaf.
\item \label{st conn Del2 2} A level 2 vertex $x$ of $H_0$ satisfies one of the following:
    \begin{enumerate}
    \item \label{st conn Del2 21} $x$ is a leaf adjacent to an endpoint of a level 1 path;
    \item \label{st conn Del2 22} $\deg(x)=2$ and $x$ is adjacent to both endpoints of a level 1 path with length
        1;
    \item \label{st conn Del2 23} $\deg(x)=3$ and $x$ is adjacent to both endpoints of a level 1 path with length 1
        and the other neighbor of $x$ is either a level 3 vertex or a level 2 vertex with degree 3 or the endpoint
        of a level 1 path.
    \end{enumerate}
\end{enumerate}
\end{prop}
\begin{figure}
\includegraphics{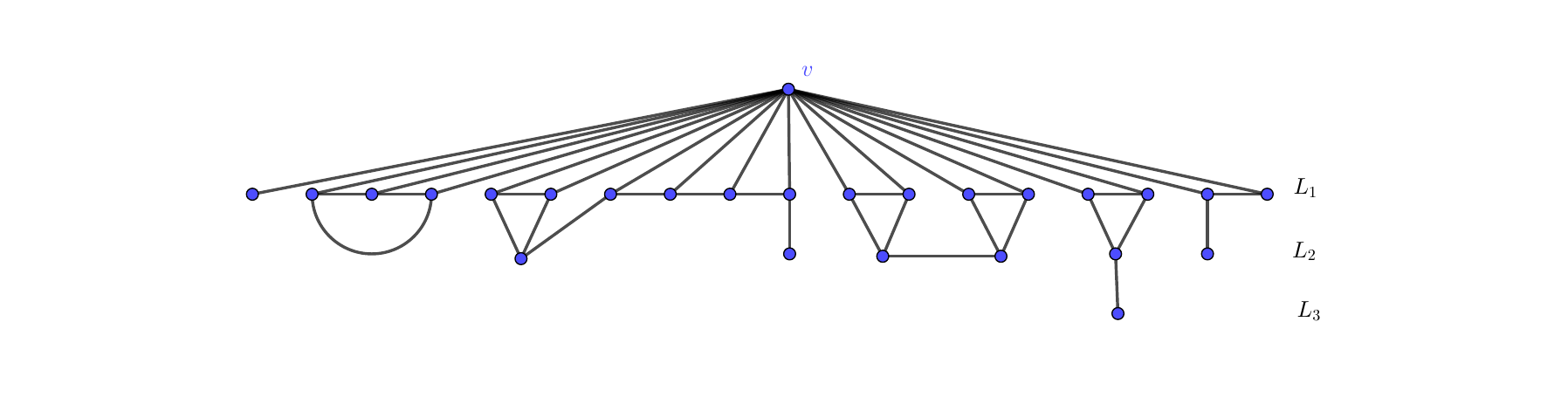}
\caption{An example of $H_0$ satisfying conditions of \pref{st conn Del2}\label{fig6}}
\end{figure}
\begin{proof}
(\give): Assume that $y$ is a level 1 isolated vertex. If $\deg(y)=2$, then $y$ is in a triangle $yva$ and hence $y$ is
adjacent to the level 1 vertex $a$, a contradiction. If $\deg(y)=3$ and $y$ is adjacent to 2 level 2 vertices $a,b$,
then there must exist a strong path $F_1, \ldots, F_k$  starting with the 3-star at $y$ and ending with a 3-star at
$v$. Then $F_2$ must be a triangle containing two edges of $F_1$, hence $F_2=\{ya, yb, ab\}$ and $F_3$ is the 3-star at
one of $a$ or $b$, say $a$. Then the level 2 vertex $a$ has degree 3 but is not adjacent to the endpoints of a level 1
path, contradicting \pref{lem st con}. Consequently, $\deg(y)=1$. So the only level 1 vertices that can be adjacent to
a level 2 vertex are the endpoints of level 1 paths and since each such vertex has two neighbors in $L_1\cup L_0$, each
can have at most one neighbor in $L_2$.

Now assume that $x$ is a level 2 vertex, according to the above argument, $x$ is adjacent to an endpoint $a$ of a level
1 path. If $\deg(x)=1$, then $x$ is a leaf and case \pref{st conn Del2 2}\pref{st conn Del2 21} holds. If $\deg(x)=2$,
then $x$ is in a triangle $axb$. If $b$ is not in $L_1$, then $a$ has two neighbors in $L_2$, contradicting the above
remarks. Thus $b$ has level 1 and is adjacent to $a$, that is $a,b$ are the endpoints of a level 1 path with length 1
as in \pref{st conn Del2 2}\pref{st conn Del2 22}. If $\deg(x)=3$, then two of the neighbors of $x$ are endpoints of a
level 1 path with length 1 according to \pref{lem st con}. The other neighbor $y$ has either level 1  and must be an
endpoint of a  level 1 path by the above argument, or has level 3, or has level 2. In the last case, since $y$ is
adjacent to $x$ and a level 1 vertex, $\deg(y)\geq 2$ and because both vertices adjacent to a level 2 vertex with
degree 2 lie in $L_1$, $\deg(y)\neq 2$. Hence if the third neighbor $y$ of $x$ has level 2, then $\deg(y)=3$, as
required.

Finally suppose that $x$ has level 3. According to \pref{lem st con}, $\deg(x)\leq 2$. If $\deg(x)=2$, then $x$ is in a
triangle containing a level 2 vertex $y$. But then $y$ has two neighbors in $L_2 \cup L_3$ which contradicts part
\pref{st conn Del2 2} proved above. Thus $\deg(x)=1$.

 (\rgive): First note that these conditions ensure that the only level 1 vertices with a neighbor in $L_2$ are
endpoints of level 1 paths and $\cup_{i=4}^\infty L_i=\tohi$. To prove that $\Delta_0^{[2]}$ is strongly connected, it
suffices to show that from every 3-star or triangle $F$ of $H_0$, there is a strong path $F=F_0, F_1, \ldots , F_t$ in
$\Delta_0^{[2]}$, where $F_t$ is a 3-star at $v$, because there exists strong paths between any two 3-stars at $v$.

Clearly every triangle containing $v$ is adjacent to a 3-star at $v$. Also a level 1 vertex $a$ with $\deg(a)=3$, is
adjacent to at least another level 1 vertex. Else $a$ is a level 1 isolated vertex and is not adjacent to any level 2
vertex by condition \pref{st conn Del2 2}, which means $\deg(a)=1$. Therefore the 3-star at $a$ is adjacent to a
triangle containing $v$. Now assume that $F$ is a triangle on which $v$ does not lie. Since every level 3 vertex is a
leaf and at most one of the neighbors of a level 2 vertex has level $\neq 1$, there must be at least one level 1 vertex
$a$ on $F$. Therefore, $F$ is adjacent to the 3-star at $a$. Finally, any 3-star with center not on $L_0\cup L_1$, must
be centered at a level 2 vertex. Then by \pref{st conn Del2 2}\pref{st conn Del2 23}, this 3-star is adjacent to a
triangle containing the two endpoints of a level 1 path of length 1. This shows that starting from any 2-face of
$\Delta_0$ and by passing from adjacent 2-faces, we can reach a 3-star at $v$.
\end{proof}


For simplicity, in the following definition we give a name to the graphs of the form $H_0$ satisfying the conditions of
the previous result. To consider the cases that the maximum degree of $H$ is at most 3, we state the definition a
little bit more general.

\begin{defn}
Suppose that $C$ is a graph, $v$ is a vertex of $C$ and $r$ is a positive integer. We say that $C$ is an $r$-graph
rooted at $v$ or simply an $r$-graph, if $C$ is connected, $\deg(v)=r$, all other vertices of $C$ have degree at most
$\min\{r,3\}$, all vertices of $C$ with degree 2 are in some triangles and also $C$ satisfies the conditions of
\pref{st conn Del2}, where the level of a vertex of $C$ is defined by $L_0=\{v\}$ and $L_i=\N(L_{i-1})\sm
(\cup_{j=0}^{i-1} L_j$).
\end{defn}

Thus \pref{st conn Del2} states that $\Delta_0^{[2]}$ is strongly connected \ifof $H_0$ is an $r$-graph for some $r>3$.

To find out when $\Delta_0^{[2]}$ is CM, we need an algebraic tool. Let $\Gamma$ be a simplicial complex and denote by
$\tl{C}_d(\Gamma)= \tl{C}_d(\Gamma; K)$ the free $K$-module whose basis is the set of all $d$-dimensional faces of
$\Gamma$. Consider the $K$-map $\rnd_d: \tl{C}_d(\Gamma) \to \tl{C}_{d-1}(\Gamma)$ defined by
 $$\rnd_d(\{v_0, \ldots, v_d\}) =\sum_{i=0}^d (-1)^i \{v_0, \ldots, v_{i-1},v_{i+1},\ldots, v_d\},$$
where $v_0<\cdots <v_d$ is a linear order. Then $(\tl{C}_\blt, \rnd_\blt)$ is a complex of free $K$-modules and
$K$-homomorphisms called the \emph{augmented oriented chain complex} of $\Gamma$ over $K$. We denote the $i$-th
homology of this complex by $\tl{H}_i(\Gamma; K)$. The Reisner theorem states that  $\Gamma$ is CM over $K$, \ifof  for
all faces $F$ of $\Gamma$ including the empty face and for all $i< \dim \link_\Gamma(F)$, one has
$\tl{H}_i(\link_\Gamma(F); K)=0$ (see \cite[Theorem 8.1.6]{hibi}). In particular, applying this with $F=\tohi$ we see
that if $\Gamma$ is CM, then for $i=\dim \Gamma-1$, we must have $\tl H_i(\Gamma; K)=0$.

\begin{thm}\label{seq CM}
Suppose that $H$ is a connected graph with at least 1 edge. Let $\Delta=\Delta(\L(H))$. Then the following are
equivalent.
\begin{enumerate}
\item \label{seq CM1} $\Delta$ is vertex decomposable.

\item \label{seq CM2} $\Delta$ is shellable.

\item \label{seq CM3} $\Delta$ is sequentially CM (over some field).

\item \label{seq CM4} For some positive integer $r$, there is an $r$-graph $H_0$ in which every level 2 vertex with
    degree 3 has a leaf neighbor and $H$ can be constructed from $H_0$ by consecutively applying the operations
    \pref{A} and \pref{B} of \pref{split}\pref{split 3}.

\item \label{seq CM5} If $H'$ is the graph obtained by splitting all vertices of $H$ with degree 2 which are not in
    any triangle, then every connected component of $H'$ is an edge except at most one. The only non-edge connected
    component of $H'$, if exists, is an $r$-graph for a positive integer $r$, in which every level 2 vertex with
    degree 3 has a leaf neighbor.
\end{enumerate}
\end{thm}
\begin{proof}
Note that if $r \leq 3$, then there are only finitely many $r$-graphs (up to isomorphism) which are either a path with
length 1 or a triangle or a 3-star or one of the graphs in Figure \ref{fig1}. Thus if the maximum degree of $H$ is at
most 3, then the result follows from \pref{no deg>4}. So we assume that $r$ and the maximum degree of $H$ are at least
4.

\pref{seq CM1} \give\ \pref{seq CM2} \give\ \pref{seq CM3} are well-known. Also the proof of \pref{seq CM4} $\iff$
\pref{seq CM5} follows the fact that``undoing'' each vertex splitting is indeed applying \pref{A} or \pref{B} of
\pref{split}\pref{split 3} which is proved in the proof of \pref{split 2} \give\ \pref{split 3} in \pref{split}.

\pref{seq CM3} \give\ \pref{seq CM4}: Suppose that $\Delta$ is sequentially CM over a field $K$ and $r\geq 4$ is the
maximum vertex degree in $H$. According to \pref{1 deg>3}, there is exactly one vertex $v$ of $H$ with degree $\geq 4$.
Further, by \pref{split} and \pref{st conn Del2} and since every CM complex is strongly connected, it follows that $H$
can be constructed from an $r$-graph $H_0$ rooted at $v$ by consecutively applying the operations \pref{A} and \pref{B}
of \pref{split}\pref{split 3} and $\Delta_0=\Delta(\L(H_0))$ is sequentially CM. Thus we just have to show that a level
2 vertex $x$ of $H_0$ with degree 3 has a leaf neighbor. Suppose it is not true. According to \pref{st conn
Del2}\pref{st conn Del2 2}, one of the following cases may occur.

\textbf{Case 1:} $x$ has a level 2 neighbor $y$ with degree 3. Assume that $a, b$ are the level 1 neighbors of $x$ and
$c$ is a level 1 neighbor of $y$. Let $e_1=va$, $e_2=vc$, $e_3=ax$, $e_4=xy$, $e_5=yc$ and $e_6=xb$ be vertices of
$\Delta_0$  with $e_i<e_j \iff i<j$ (see Figure \ref{fig7}(a)). Consider $\sigma= \{e_3,e_4\}+ \{e_4,e_5\}-
\{e_2,e_5\}- \{e_1,e_2\}+ \{e_1,e_3\}\in \tl C_1(\Delta_0^{[2]})$. Then $\rnd_1(\sigma)=0$ and as $\tl
H_1(\Delta_0^{[2]};K)=0$, we must have $\sigma\in \im\rnd_2$, that is, there is a $\delta\in \tl C_2(\Delta_0^{[2]})$
with $\rnd_2(\delta)=\sigma$. Since $\{e_3,e_4\}$ appears only in the 2-face $F=\{e_3,e_4,e_6\}$, the coefficient of
$F$ in $\delta$ must be 1. But then in $\rnd_2(\delta)$ the coefficient of $\{e_4,e_6\}$ is 1, because $F$ is the only
2-face containing $\{e_4,e_6\}$. This contradiction shows that case 1 cannot occur.

\textbf{Case 2:} The third neighbor $y$ of $x$ has level 1. Again assume that $a,b$ are the other level 1 vertices
adjacent to $x$ and set $e_1=va$, $e_2=vy$, $e_3=ax$, $e_4=xy$ and $e_6=xb$ with $e_i<e_j \iff i<j$ (see Figure
\ref{fig7}(b)). By a similar argument as in case 1 with $\sigma=\{e_3,e_4\}-\{e_2,e_4\}-\{e_1,e_2\}+\{e_1,e_3\} \in
\ker \rnd_1$, we again reach a contradiction.

\begin{figure}
\center\includegraphics[scale=1.6]{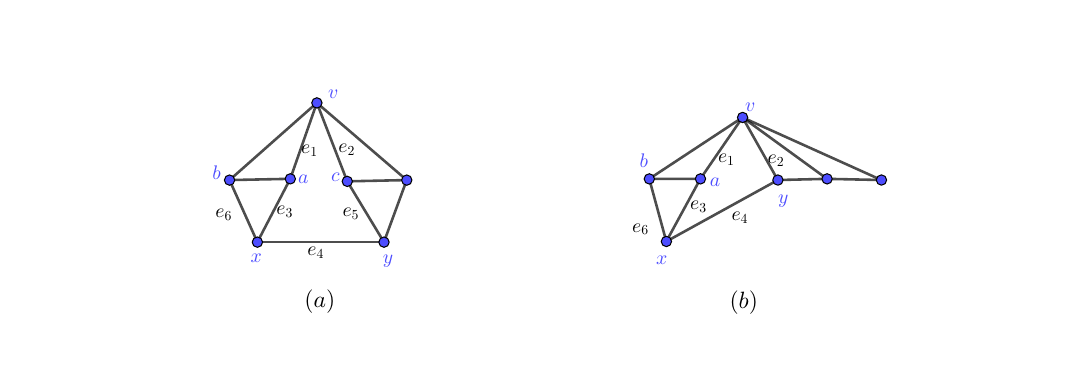}
\caption{Illustrations for the proof of Theorem \ref{seq CM}; (a) is case 1 and (b) is case 2 \label{fig7}}
\end{figure}

\pref{seq CM4} \give\ \pref{seq CM1}: We show that $\Delta_0=\Delta(\L(H_0))$ is vertex decomposable. Then the result
follows from the ``moreover'' statement of \pref{split}. For this, we use induction on the number of edges of $H_0$.
Assume that $H_0$ is rooted at the vertex $v$. First we show that for every vertex $e$ of $\Delta_0$ not incident to
$v$, $\link_{\Delta_0}(e)$ is vertex decomposable. Indeed, if $e$ is in a triangle of $H_0$, whose other two edges are
$f_1$ and $f_2$, then every edge of $H_0$ adjacent to $e$ in $\L(H_0)$, is also adjacent to either $f_1$ or $f_2$ and
since $f_1$ and $f_2$ are adjacent, the pure 1-dimensional complex $\link_{\Delta_0}(e)$ is connected and by
\pref{graph ver dec}, vertex decomposable. If $e$ is not in any triangle, then the conditions of part \pref{seq CM4},
ensures that $e$ is a leaf edge and so again $\link_{\Delta_0}(e)$ is connected and vertex decomposable.

Assume that $H_0$ has a level 3 vertex $x$. Then $x$ is a leaf and is adjacent to a level 2 vertex $y$ which is in a
triangle. Let $e=xy$. Then the only facet of $\link_{\Delta_0}(e)$ is strictly contained in this triangle and hence $e$
is a shedding vertex of $\Delta_0$. Also if $H_0'=H_0-e-x$, then $\Delta_0-e=\Delta(\L(H_0'))$ and $H_0'$ satisfies the
conditions of \pref{seq CM4}.  So $\Delta_0-e$ is vertex decomposable by induction hypothesis and hence $\Delta_0$ is
also vertex decomposable. Thus we can assume $H_0$ has no level 3 vertex. By a similar argument we can also assume that
no level 2 vertex of $H_0$ is a leaf.

Since every level 2 vertex with degree 3 is adjacent to a level 3 vertex which we have assumed does not exist, if $x$
is a level 2 vertex, then $\deg(x)= 2$. So by the assumptions of \pref{seq CM4}, $x$ is adjacent to the endpoints $a,b$
of a level 1 path with length 1. Let $e=xa$. Then the facets of $\link_{\Delta_0}(e)$ are $\{va, ab\}$ and $\{xb, ab\}$
which are contained in the facets $\{va, ab, vb\}$ and $\{xb, ab, vb\}$ of $\Delta_0-e=\Delta(\L(H_0-e))$,
respectively. Because $H_0-e$ satisfies the conditions of \pref{seq CM4}, $\Delta_0-e$ is vertex decomposable by
induction hypothesis and hence so is $\Delta_0$. Therefore, we can assume that $H_0$ has no vertex with level 2.

Let $C$ be a level 1 cycle of $H_0$ with two adjacent vertices $x$ and $y$. Assume that in $C$, $a\neq y$ is adjacent
to $x$ and $b\neq x$ is adjacent to $y$ (we may have $a=b$). Let $e=xy$. The facets of $\link_{\Delta_0}(e)$ are $\{vx,
xa\}$, $\{vy, yb\}$, $\{vx, vy\}$ and if $a=b$, $\{xa,ay\}$, which are in the following facets of $\Delta_0-e$,
respectively: $\{vx, xa, va\}$, $\{vy, yb,vb\}$, the $r$-star at $v$ and the 3-star at $a$. Thus again it follows that
$e$ is a shedding vertex of $\Delta_0$ and $\Delta_0-e$ is vertex decomposable by induction hypothesis and hence
$\Delta_0$ is also vertex decomposable.

If $P$ is a level 1 path of $H_0$ starting with the vertices $a$ and $b$, then similar to the above paragraphs one can
see that $e=ab$ is a shedding vertex with $\Delta_0-e$ vertex decomposable and it follows that $\Delta_0$ is vertex
decomposable. If $H_0$ has no level 1 cycle or level 1 path, then it is just an star centered at $v$. In this case
$\Delta_0$ is a simplex and the result follows.
\end{proof}

From this theorem we see that if $H$ is the graph in Figure \ref{fig6} and $\Delta=\Delta(\L(H))$, then $\Delta$ is not
sequentially CM, that is, $\Delta^{[2]}$ is not CM although it is strongly connected. Also as another example, Theorem
\ref{seq CM} shows that the line graph of the graph in Figure \ref{fig8} has a sequentially CM clique complex.

\begin{figure}
\begin{center}
\includegraphics{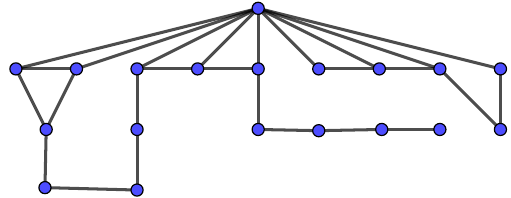}
\caption{A graph whose line graph has a sequentially CM clique complex \label{fig8}}
\end{center}
\end{figure}

                            \subsection*{An algorithm}

At the end of this paper, we show that using Theorem \ref{seq CM}, we can present a linear time algorithm which takes
as input a graph $G$ and checks whether $G$ is a line graph or not and if yes, says whether $\Delta(G)$ is sequentially
CM. Checking if $G$ is a line graph and even returning an $H$ such that $G=\L(H)$ has been previously done by Lehot in
\cite{lehot} in a linear time. Thus we can assume that $H$ is given and we must find out if $\Delta(\L(H))$ is
sequentially CM. We assume that $H$ is given as lists of neighbors of vertices and $n$ is the number of vertices of
$H$. Here we state an algorithm, the correctness of which is ensured by Theorem \ref{seq CM} and show that its worst
case time complexity is $\Theta(n)$. In this algorithm, we use breadth-first search (BFS) which can be found in for
example \cite{west}.

\paragraph{Step 1:}
Run through the vertices of $H$ and compute the degree of each vertex. If for a second time a vertex with degree more
than three is visited, return false. Also for each vertex $x$ with degree 2 and with neighbors $a$ and $b$, check if
$a$ is a neighbor of $b$. If not, split the vertex $x$ by removing the edge $xb$ and adding a new vertex adjacent only
to $b$.

Note that if $H$ has more than one vertex with degree $>3$, then after at most checking $3(n-1)+r+1\leq 4n$ edges where
$r$ is the degree of the first vertex with degree $>3$ that we encounter, we will find out, return false and exit. Else
the number of edges is at most $3n/2$ and this step is carried out with $\Theta(n)$ time complexity. Also note that the
obtained graph has at most $2n$ vertices.

\paragraph{Step 2:}
Compute the connected components of the obtained graph (say, by BFS). If more than one connected component is not an
edge return false. If all connected components are edges, return true. Else let $H_0$ be the only connected component
which is not an edge. This step clearly needs $\Theta(n)$ time complexity.

\paragraph{Step 3:}
Find a vertex $v$ with maximum degree in $H_0$. Run a BFS starting at $v$ and mark each visited vertex with its level
which is the distance of the vertex from $v$. When visiting a level 2 vertex $y$ consider the following cases.
\begin{description}
\item[$\deg(y)=1$:] Let $a$ be the neighbor of $y$ (which has level 1). If $a$ has no level 1 neighbor (so that $a$
    is not the endpoint of a level 1 path), return false.
\item[$\deg(y)=2$:] The neighbors of $y$ should have level 1 and be adjacent. If not, return false.
\item[$\deg(y)=3$:] Then its neighbors should be two level 1 adjacent vertices and a vertex not yet visited. If not,
    return false.
\end{description}
Also when visiting a level 3 vertex $x$, if $x$ has not degree 1, return false. This step also consumes $\Theta(n)$
running time.
\paragraph{Step 4:} Return true.

Note that steps 2 and 3 can be carried out simultaneously.

We have to make two remarks regarding the correctness of the algorithm. First, if the maximum degree of $H$ is 3 and
$\Delta$ is sequentially CM, then the graph $H_0$ is a 3-star or one of the graphs in Figure \ref{fig1}. All of these
graphs are 3-graphs rooted at $v$, where $v$ can be any of the degree 3 vertices. Thus in step 3 it does not differ
which vertex with degree 3 we choose as $v$. Furthermore in step 3, we may visit a level 2 vertex with degree 3 which
has a level 2 neighbor $z$ not still visited, without returning false at that moment. In this case, when visiting the
vertex $z$, since $z$ has a visited level 2 neighbor, the algorithm returns false.

                            
\end{document}
`